\documentclass[leqno,10pt]{amsart}
\usepackage{latexsym,amssymb,amsmath,graphics,color}
\usepackage{natbib,enumerate}
\usepackage[dvips]{graphicx}

\makeatletter
\@namedef{subjclassname@2010}{%
  \textup{2010} Mathematics Subject Classification}
\makeatother

\newtheorem{thm}{Theorem}[section]

\newtheorem{lem}[thm]{Lemma}

\theoremstyle{definition}

\newtheorem{rem}[thm]{Remark}

\numberwithin{equation}{section}
\newcommand{\8}{\infty}

\usepackage{color}

\newcommand{\E}{\mathbb{E}}
\newcommand{\Prob}{\mathbb{P}}

\newcommand{\Pfs}[1][]{\ensuremath{\mathbb{P}_{#1}\text{-a.s.}}}

\newcommand{\N}{\mathbb{N}}

\newcommand{\R}{\mathbb{R}}

\newcommand{\A}{\mathcal{A}}
\newcommand{\F}{\mathcal{F}}
\renewcommand{\S}{\mathcal{S}}

\newcommand{\llam}{{l}}

\newcommand{\supp}{\mathrm{supp}\,  }

\renewcommand{\epsilon}{\varepsilon}
\renewcommand{\rho}{\varrho}

\newcommand{\1}[1][]{\mathbf{1}_{#1}}
\newcommand{\norm}[1]{\ensuremath{\left\| {#1} \right\|}}

\newcommand{\abs}[1]{\ensuremath{\left| {#1} \right|}}
\newcommand{\skalar}[1]{\langle #1 \rangle}

\newcommand{\eqdist}{\stackrel{\mathcal{L}}{=}}

\newcommand{\Step}[1][]{\textsc{Step #1}}

\newcommand{\Pset}[1][]{\mathcal{P}_{#1}}

\newcommand{\law}[1]{\mathcal{L}\left( #1 \right)}

\newcommand{\Rd}{\R^d}
\newcommand{\Rdnn}{{\R^d_\ge}}

\newcommand{\Cf}[2][]{\mathcal{C}^{#1}\left( #2 \right)}

\newcommand{\U}[1][]{\mathbb{U}_{#1}}
\newcommand{\condC}{(C)}

\newcommand{\Sp}{\mathbb{S}_\ge}
\newcommand{\Sd}{\mathbb{S}}
\newcommand{\Mset}{\mathcal{M}_+}
\newcommand{\interior}[1]{\mathrm{int}({#1})}
\newcommand{\est}[1][s]{r^*_{#1}}

\newcommand{\es}[1][s]{r_{#1}}
\newcommand{\Ps}[1][s]{P^{#1}}

\newcommand{\nus}[1][s]{\nu_{#1}}
\newcommand{\mM}{\mathbf{M}}

\newcommand{\mPi}{\matrix{\Pi}}

\newcommand{\mA}{\matrix{A}}
\renewcommand{\matrix}[1]{\mathbf{#1}}

\newcommand{\as}{\cdot}

\newcommand{\T}{T}

\newcommand{\closure}[1]{\overline{#1}}

\renewcommand{\P}[2][]{\ensuremath{\mathbb{P}_{#1} \left( {#2} \right)}}

\renewcommand{\mPi}{\Pi}
\renewcommand{\mA}{A}
\renewcommand{\mM}{M}

\newcommand{\ip}{\textit{(i-p)}}
\newcommand{\ipo}{\textit{(i-p,o)}}
\newcommand{\ide}{\textit{(id)}}

\newcommand{\wt}{\widetilde}
\renewcommand{\T}{{\mathbb T}}
\newcommand{\W}{{\mathbb W}}
\renewcommand{\F}{{\mathcal F}}
\newcommand{\bU}{\mathbb U}
\newcommand{\bi}{\mathbf i}
\newcommand{\bj}{\mathbf j}

\newcommand{\bo}{\mathbf 1}

\begin{document}





\title[Tail Asymptotics for Fixed Points of Smoothing Transformations]{Precise Tail Asymptotics for Attracting Fixed Points of Multivariate Smoothing Transformations - Complete Proof}

    \author{Dariusz Buraczewski$^\star$, Sebastian Mentemeier$^\dagger$}
 \address{ $^\star$Uniwersytet Wroc\l awski \\ Instytut Matematyczny \\ pl. Grunwaldzki 2/4 \\ 50-384 Wroc\l aw, Poland \\
 $^\dagger$ TU Dortmund \\ Fakult\"at f\"ur Mathematik\\ Vogelpothsweg 87 \\ 44227 Dortmund, Germany}
 \email{$^\star$ dbura@math.uni.wroc.pl, $^\dagger$ sebastian.mentemeier@tu-dortmund.de}


\begin{abstract}

 Given $d \ge 1$, let   $(A_i)_{i\ge 1}$ be a sequence of random $d\times d$ real matrices and $Q$ be a random vector in $\R^d$. We
consider fixed points of multivariate smoothing transforms, i.e. random variables $X\in \Rd$ satisfying
$$
X \text{ has the same law as } \sum_{i \ge 1} A_i X_i + Q,
$$
where $(X_i)_{i \ge 1}$ are i.i.d.~copies of $X$ and independent of $(Q, (A_i)_{i \ge 1})$.
The existence of fixed points that can attract point masses can be shown by means of contraction arguments.
%
Let $X$ be such a fixed point. Assuming that the action of the matrices is expanding as well as contracting with positive probability, it was shown in a number of papers that there is $\beta >0$  with
$\lim_{t \to \infty} t^\beta \P{\skalar{u,X}>t} ~=~ K\cdot f(u)$, where
$u$ denotes an arbitrary element of the unit sphere and $f$ a positive function and $K \ge 0$. However in many cases it was not established that $K$ is indeed positive.

In this paper, under quite general assumptions, we prove that
$$
\liminf_{t\to\8} t^{\beta} \P{\langle u,X \rangle > t}> 0,
$$
completing, in particular, the results of  \cite{Mirek2013} and \cite{BDMM2013}.%
\end{abstract}



\thanks{
D.~Buraczewski was partially supported by the National Science Centre, Poland (Sonata Bis, grant number DEC-2014/14/E/ST1/00588). The main part of this work was done while S.~Mentemeier held a post-doc position at the Mathematical Institute,  University of Wroc\l aw.
}
\maketitle

\section{Introduction}
\subsection{The (multivariate) smoothing transform}
Let $d \ge 1$. Let $(Q,(A_i)_{i \ge 1})$ be a random element of $\R^d \times M(d \times d, \R)^\N$, that is $Q$ is a random vector and $(A_i)_{i\ge 1}$ is a sequence of random matrices. We assume that the random number
$ N ~:= \max\{ i \, : \, A_i \neq 0\}$
is finite a.s.
If $X \in \Rd$ is a random variable such that
\begin{equation} X \text{ has the same law as } \sum_{i=1}^N A_i X_i + Q, \label{eq:SFPE} \end{equation}
where $(X_i)_{i \ge 1}$ are i.i.d.~copies of $X$ and independent of $(Q, (A_i)_{i \ge 1})$, then we call the law $\law{X}$ of $X$ a fixed point of the (multivariate, if $d >1$) {\em smoothing transform}. By a slight abuse of notation, we also call $X$ itself a fixed point.

Eq. \eqref{eq:SFPE}  has drawn a lot of attention for decades. In the univariate case this equation occurs in  various areas, e.g.
the analysis of  recursive algorithms (\cite{Roe1991,R2001,NR2004}), branching particle systems (\cite{DL1983}),
 Google's  PageRank algorithm (\cite{JO2010a,JO2010b,chen2014}).
Also the multivariate situation  draws a lot of  attention. A classical example where \eqref{eq:SFPE} appears is  the joint distribution of
key comparisons and key exchanges for Quicksort  (\cite{NR2004}). However in this case the action of the matrices is purely contracting, and therefore all fixed points have exponential moments, which is not in the scope of the present paper. Some recent examples are related to kinetic models, see \cite{Bassetti2014}. Then solutions to \eqref{eq:SFPE} describe e.g. equilibrium distribution of the particle velocity in Maxwell gas.

The aim of this paper is to describe the tail behavior of fixed points, i.e. the decay rate of $$\P{\abs{X}>t}\qquad \mbox{or} \qquad \P{\skalar{u,X}>t},$$
 as t goes to infinity, where
$u$ denotes an arbitrary element of the unit sphere $\Sd$. Our focus is on the multivariate case, where we resolve the open question, whether the limit $\lim_{t \to \infty} \P{\abs{X}>t}$ is indeed positive.

\subsection{Univariate smoothing transform}
In dimension $d=1$, complete results about the structure of fixed points are available under very weak assumptions, see \cite{DL1983,Liu1998,Biggins1997,ABM2012,Alsmeyer2013, BK2014} for the case of $A_i \ge 0$, and \cite{Iksanov2014} for the most general case of $A_i \in \R$. It turns out, that the characterization depends on the function
\begin{equation} \label{eq:ms} m(s):= \E \sum_{i=1}^N \abs{A_i}^s, \end{equation}
which is log-convex, and in particular on the value $\alpha = \inf\{ s > 0 \, : \, m(s) = 1\}$.
It is shown that there are two classes of fixed points: Fixed points are either mixtures of $\alpha$-stable laws and attract (only) laws with $\alpha$-regular varying tails, or have a finite moment of order $\alpha + \epsilon$ for some $\epsilon >0$ (subject to the assumption $\E \abs{Q}^{\alpha+\epsilon} < \infty$) and attract point masses.The ``relevant'' solutions in most situations are those from the second class, which we call {\em attracting fixed points} in the sequel, and it is therefore important to investigate their properties, such as tail behavior.
Under various assumptions on $(Q, (A_i)_{i \ge 1})$ and $N$, it has been shown in \cite{Gui1990,Liu2001,JO2010a,JO2010b} that if, roughly speaking, there is $\beta>\alpha$ with $m(\beta)=1$ and $\E \abs{Q}^\beta < \infty$ (including the case $Q \equiv 0$), then
\begin{equation}
\label{eq: limit}
\lim_{t \to \infty} t^\beta \P{\abs{X}>t} ~=~ K \ge 0 .
\end{equation}
There is also a rich literature concerning the case when $\alpha$ is the unique point such that $m(\alpha)=1$, that is when $m'(\alpha)=0$ (see \cite{DL1983, Liu1998, BK2005, B2009.SPA, BK2014}).

It is obviously a very important question, whether $K$ is indeed positive, since otherwise, $t^\beta$ might be not the precise rate.
For $A_i \ge 0$ and $Q=0$, positivity of $K$ is proved in \cite{Gui1990,Liu2001}, but it remained - except for some special cases - an open question in \cite{JO2010a,JO2010b}, where the cases $Q \neq 0$ resp. $A_i \in \R$ were considered. This question was  answered  in the consecutive papers \cite{Alsmeyer2013a} by using complex function arguments, \cite{Jelenkovic2014} (only for the inhomogeneous case) and in \cite{BDZ2014+} by using large deviation estimates (but for i.i.d. $(A_i)$ only).

\subsection{What is this paper about?}
In the multivariate setting $d>1$, an analogue of the function $m$ can be defined (details given below), and in the case where the $A_i$, $Q$ and $X$ all have nonnegative entries, it has been shown in \cite{Mentemeier2013}, that again fixed points are either mixtures of multivariate $\alpha$-stable laws (with $\alpha$ defined as before), or have a finite moment of order $\alpha + \epsilon$, if $\E \abs{Q}^{\alpha +\epsilon} < \infty$.

Tail behavior of attracting fixed points in this setting has been analyzed for the case $\alpha=1$ and $Q=0$ in \cite{BDGM2014}, where it has been shown that
\begin{equation}\label{eq:lim} \lim_{t \to \infty} t^\beta \P{\skalar{u,X}>t} ~=~ K r(u)  \end{equation} with a positive continuous function $r$ on $\Sp := \Sd \cap [0,\infty)^d$ for $\beta$ being the unique value such that $\beta >\alpha$ and $m(\beta)=1$. In this case, also positivity of $K$ has been proved.

The inhomogeneous case $Q \neq 0$, with $\alpha \le 1/2$, has been studied in \cite{Mirek2013} and the existence of the limit in Eq. \eqref{eq:lim} is proved there, but it remained an open question, whether $K$ is positive (at least for $\beta <1$).

The case of invertible matrices $(A_i)_{i \ge 1}$ was studied in \cite{Bassetti2014} and \cite{BDMM2013}, with tail behavior being studied mainly in the latter paper. There once more existence of the limit in Eq. \eqref{eq:lim} was proved, but not the positivity of $K$.

\medskip

The contribution of this paper is to prove the positivity of $K$ in all these cases.
We will follow the strategy developed in \cite{BDZ2014+}, getting rid at the same time  of some of its restrictions. We use this approach since we were unable to find an extension of the method used in \cite{Alsmeyer2013a} to the multivariate setting, and since the proof in \cite{Jelenkovic2014} relies on the assumption $\P{|Q|>0}>0$.

The main technical ingredient is a large deviation result for products of random matrices, which was developed in \cite{BM2013} and provided for several classes of random matrices.
In the next section we introduce the three classes of matrices which are considered in this paper and further notation relevant for the multivariate case. The main results of this paper are formulated in Section \ref{sec:main results}. The remaining sections are devoted to the proof.

\section{Notations}
\label{sec: notation}
In this section, we describe, in an abbreviated form, but similar to \cite{BM2013} three sets of assumptions for random matrices, namely condition $\condC$ for nonnegative matrices and conditions $\ip$~and $\ide$~for invertible matrices. Each set of assumptions guarantees precise large deviation estimates extending the Furstenberg-Kesten-theorem (\cite{Furstenberg1960}), i.e. the SLLN for the norm of products of random matrices. These large deviation estimates will play a prominent role in our proof below.
\newcommand{\mm}{\matrix{m}}
Let $d \ge 1$. Given a probability law $\mu$ on the set of $d \times d$-matrices $M(d \times d, \R)$, let $(M_n)_{n \in \N}$ be a sequence of i.i.d.~random matrices with law $\mu$. 
Equip $\Rd$ with any norm $\abs{\cdot}$, write $\norm{\mm}:=\sup_{x \in \Sd} \abs{\mm x}$ for the operator norm of a matrix $\mM$  and denote the unit sphere in $\Rd$ by $\Sd$. We write
$$ \mm \as x := \frac{\mm x}{\abs{\mm x}}, \qquad x \in \Sd$$ for the action of a matrix $\mm$ on $\Sd$ (as soon as this is well defined). If $\Sd$ is invariant under the action of $M$, we introduce a \emph{Markov random walk} $(U_n, S_n)_{n \in \N}$ on $\Sd \times \R$ by
\begin{equation} \label{unsn} U_n := \mM_n \cdots \mM_1 \as U_0, \qquad S_n:= \log \abs{\mM_n \cdots \mM_1 U_0} = \log \abs{\mM_n U_{n-1}} + S_{n-1}, \end{equation}
for some initial data $U_0 \in \Sd$, the value of which we note by the convention $\P[u]{U_0=u}=1$, $u \in \Sd$.

Below, the following concepts will appear several times: Write $\Gamma:= [\supp \mu]$ for the semigroup of matrices, generated by the support of $\mu$. A matrix $\mm$ with an algebraic simple dominant eigenvalue $\lambda_\mm$, that exceeds all other eigenvalues in absolute value, will be called {\em proximal}, and we will denote by $v_\mm^{\pm} \in \Sd$ the corresponding normalized eigenvectors ($v_{\mm}^+=-v_\mm^-$), using the convention that $\min\{i : (v_{\mm}^+)_i > 0 \} <  \min\{i : (v_{\mm}^-)_i > 0 \}$.
Note that a matrix with all entries positive is proximal by the Perron-Frobenius theorem, and that $v_{\mm}^+$ is the Perron-Frobenius eigenvector.

\subsection{Invertible Matrices: Condition \ip}
The condition \ip~(irreducible and proximal), described below, is due to Guivarc'h, Le Page and Raugi and was studied in detail in several articles by these authors, the most comprehensive one of which is \cite{Guivarch2012}.

Let now $\mu$ be a probability measure on the group $GL(d,\R)$ of invertible $d \times d$ matrices.
Then the measure $\mu$ is said to satisfy condition \ip, if
\begin{enumerate}
\item There is no finite union $\mathcal{W}=\bigcup_{i=1}^n W_i$ of subspaces $0 \neq W_i \subsetneq \R^d$ which is $\Gamma$-invariant, i.e. $\Gamma \mathcal{W}=\mathcal{W}$. (\emph{strong irreducibility})
\item $\Gamma$ contains a proximal matrix. (\emph{proximality})
\end{enumerate}

It may happen that there is a $\Gamma$-invariant proper closed convex  cone $C$. This situation is very similar to the case of nonnegative matrices, see \cite{BDGM2014}. Therefore, we will exclude it and only consider matrices satisfying
\begin{equation}
\label{ipo} \tag{\textit{i-p,o}} \mu \text{ satisfies \ip, and there is no $\Gamma$-invariant proper closed convex  cone.}
\end{equation}

In this case, it can be shown that the Markov chain $(U_n)$ has a unique invariant probability measure, which is supported on
$$ V(\Gamma) := \closure{\left\{v_\mm^{\pm} \in \Sd \, : \, \mm \in \Gamma \text{ is proximal }\right\}}, $$
and due to the strong irreducibility, the orthogonal space of $V(\Gamma)$ is $\{0\}$.
%
Finally, write $$ \iota(\mm)~:=~ \inf_{x \in \Sd} \abs{\mm x} ~=~ \norm{\mm^{-1}}^{-1}.$$
%

\subsection{Nonnegative Matrices: Condition $\condC$}
Next, we introduce a condition on nonnegative matrices, i.e. all entries greater or equal to zero, which do not need to be invertible. We will use similar notation as for condition \ip, in order to highlight connections. Note that these assumptions can be formulated more generally for matrices leaving invariant a proper closed convex cone, see \cite{BDGM2014}.

 Denote the cone of vectors with nonnegative entries by $\Rdnn$
 and write
$ \Sp = \{ x \in \Rdnn \, : \, \abs{x}=1\}$
for its intersection with the unit sphere.
It is invariant under the action of \emph{allowable} matrices, i.e. matrices having nonnegative entries and no zero row nor column. For an allowable matrix, the quantity
$$ \iota(\mm) := \min_{x \in \Sp} \abs{\mm x} $$
is strictly positive and is the suitable substitute for $\iota$ as defined for invertible matrices.

We say that a  probability measure $\mu$ on nonnegative matrices satisfies condition $\condC$, if:
\begin{enumerate}
\item Every $\mm \in \supp \mu$ is allowable.
\item $[\supp \mu]$ contains a matrix all entries of which are strictly positive (a {\em positive matrix}).
\end{enumerate}

Once again, this guarantees the existence of a unique invariant probability measure for $(U_n)$ on $\Sp$, which is supported in
$$ V(\Gamma) := \closure{\left\{v_\mm \, : \, \mm \in \Gamma \text{ is a positive matrix }\right\}}. $$

Note that for nonnegative matrices, being positive is a stronger assumption than proximality, for it also asserts irreducibility: A diagonal matrix might be allowable and proximal as well, but in contrast to a positive matrix, its dominant eigenvector is not attractive on the whole set $\Sp$. This is why no assumption on invariant subspaces is needed here.
Instead, we have to impose an additional non-lattice condition for $(S_n)$, which is automatically satisfied under \ip:
%
Define $$ S(\Gamma) := \{\log \lambda_\mm \ : \ \mm \in \Gamma \cap \interior{\Mset} \}.$$
Then we say that $\mu$ is {\em non-arithmetic}, if the (additive) subgroup of $\R$ generated by $S(\Gamma)$ is dense.

\subsection{Invertible Matrices: Condition \ide}
The third set of assumptions, called \ide~for irreducible and density, appears first at the end of \cite{Kesten1973} and was elaborated  in \cite{AM2010}. In fact, it can be shown to imply condition \ipo. Due to the stronger assumption that $\mu$ is absolutely continuous, it often allows for simpler proofs, this is why we include it as an extra set of assumptions.

A probability measure $\mu$ on $GL(d,\R)$ 
is said to satisfy condition $\ide$ if
\begin{enumerate}
\item for all open $B \subset \Sd$ and all $x \in \Sd$, there is $n \in \N$ such that $\P{\mPi_n \as x \in B} >0$, and 
\item there are a matrix $\mm_0 \in GL(d,\R)$, $\delta, c >0$ and $n_0 \in \N$ such that
$$ 
 \P{\mPi_{n_0} \in d\mm } ~\ge~ c \1[{B_\delta(\mm_0)}](\mm) \, \llam(d \mm) ,$$
where $\llam$ denotes the Lebesgue measure on $\R^{d^2} \simeq M(d \times d,\R)$. 
\end{enumerate}
The classical example is $\mu$ having a density about the identity matrix.

It is shown in \cite[Lemma 5.5]{AM2010} that $U_n$ is a  Doeblin chain under condition \ide. The support of its  stationary probability measure is $\Sd$ by \cite[Proposition 4.3]{AM2010}, therefore  in the case of \ide\ we  have  $V(\Gamma)=\Sd$.

\subsection{Markov random walk and change of measure}
Below, we identify $\S=\Sp$ in the case of nonnegative matrices  and $\S=\Sd$ in the case of \ip- or \ide-matrices. Given a measure $\mu$ on matrices as before, set
$$ I_\mu := \{ s \ge 0 \, : \, \E \norm{M}^s  <\infty \}.$$
Then, for $s \in I_\mu$, we define  operators  in the set $\Cf{\S}$ of continuous functions on $\S$ by
\begin{equation}\label{eq:Ps}  \Ps f(x) := \E \big[ \abs{M x}^s \, f(M \as x)\big] 
\end{equation}

It was proved in \cite{Kesten1973,BDGM2014} for nonnegative matrices, in \cite{Guivarch2012} for invertible matrices under condition \ipo~ and in \cite{Mentemeier2013a} under condition \ide, that the spectral radii of these operators are given by the log-convex  and differentiable function 
\begin{equation}
\label{eq: ks}
k(s) := \lim_{n \to \infty} \left(\E{\norm{\mM_n \ldots \mM_1 }^s} \right)^\frac{1}{n}
\end{equation}
and that for each $s \in I_\mu$  there are
\begin{itemize}
\item an unique normalized  function $\es \in \Cf{\S}$;
  \item an unique probability measure $\nus \in \Pset(\S)$ satisfying
\begin{equation}\label{eigenfunction} \Ps \es = k(s) \es \quad \text{ and } \quad \Ps \nus = k(s) \nus. 
\end{equation}
\end{itemize}
Moreover, the function $\es$ is strictly positive and $\bar{s}:=\min\{s,1\}$-H\"older continuous. The support of the measure is given by $\supp\, \nus = V(\Gamma)$.
Equation \eqref{eigenfunction} yields that if $k(\gamma)=1$, then $h(u,t):=e^{\gamma t} \est[\gamma](u)$ is an harmonic function for the Markov chain $(U_n, S_n)$. Using the idea of Doob's $h$-transform, one can introduce  new probability measures $\Prob_x^\gamma$, and it turns out that under $\Prob_u^\gamma$, $S_n$ has drift $k'(\gamma)$, i.e.
\begin{equation}
\label{SLLN} \lim_{n \to \infty} \frac{S_n}n ~=~ \frac{k'(\gamma)}{k(\gamma)} \quad \Prob_x^\gamma\text{-a.s.}
\end{equation}
This idea can be extended (see \cite[Section 2]{BM2013} for details) to yield {\em exponentially shifted} probability measures $\Prob_x^\gamma$ for all $\gamma \in I_{\mu}$, such that the property \eqref{SLLN} holds.

We will make use of the following estimate: We obtain from \cite[Corollary 4.6]{BDGM2014} for Condition \condC, \cite[Lemma 2.8]{Guivarch2012} for condition \ipo~ (the proof working for \ide~ as well) that for all $s \in I_\mu$ there is a constant $c_s$, independent of $n$, such that
\begin{equation} \E[\norm{\Pi_n}^s] ~\le~ c_s k(s)^n. \label{eq:kEPi}\end{equation}

\subsection{Case distinction concerning $N$}

We will make the following case distinction concerning the number $N$:
\begin{enumerate}
\item[(N-random)] $N \in \N$ is random with $1 < \E N < \infty$, and conditioned upon $N$, $(\mA_i)_{i= 1}^N$ are i.i.d.~ with law $\mu$, and the variables $Q$ and $(N, (\mA_i)_{i\ge 1})$ are independent. 
\item[(N-fixed)] $N \ge 2$ is fixed, $(\mA_1, \dots, \mA_N, Q)$ having any dependence structure, and $ \bigcup_{i=1}^N \supp A_i$ is bounded
\end{enumerate}
The case (N-fixed), without any loss of generality, can be reduced to the situation where all the random variables $\mA_1,\ldots,\mA_N$ are identically distributed. This is achieved by replacing $A_i$ with $A_{\tau(i)}$, where $\tau$ is a random permutation, independent of $(A_i)_{i=1}^N$, distributed uniformly on the symmetry group of $\{1, \dots, N\}$, see \cite[Proposition A.1]{BDMM2013} for more details.
The same argument can be used to supply the even stronger property that $A_1, \dots, A_N$ are exchangeable, i.e., for any (vector valued) function $f$ on $M(d\times d,\R)^N$ and any permutation $\sigma$ of $\{1,\ldots, N\}$ it holds that
\begin{equation}
f(A_1,\ldots, A_N) \eqdist f(A_{\sigma(1)},\ldots, A_{\sigma(N)}).
\end{equation}
This property also follows immediately (for any permutation of a finite index set), if $(A_i)_{i \ge 1}$ are i.i.d.
Hence for both cases, (N-random) and (N-fixed), we can now introduce the following standing assumption:
\begin{equation} \tag{StA} \text{ $(A_i)_{i=1}^N$ are identically distributed and exchangeable.} \label{eq:perm} \end{equation}
We then set
$$ \mu ~:=~\law{\mA_1^* \in \cdot} ,$$ i.e. $\mu$ is the law of the transpose of $A_1$.
Then the general, multivariate version of the function $m(s)$ (see Eq. \eqref{eq:ms}) is given by 
$$ m(s) := (\E N) k(s),$$
with $k(s)$ as defined in \eqref{eq: ks} (with $M_1, M_2, \dots$ being i.i.d.~ random variables, having law $\mu$).

\section{Statement of Results}
\label{sec:main results}

Here is our main result in the multidimensional situation:
\begin{thm}\label{thm:d>1}
Assumptions:
\begin{enumerate}
\item Let {\em either} (N-random) {\em or} (N-fixed) be satisfied.
\item Geometrical assumptions: 
Assume {\em one} of the following 
\begin{enumerate}[(Ga)]
\item $\mA_i$ and $Q$ are nonnegative, $\mu$ satisfies $\condC$ and is nonarithmetic,  or
\item $\mA_i$ are invertible and satisfy \ipo~or \ide.
\end{enumerate}
\item Moment assumptions: Assume {\em all} of the following
\begin{enumerate}[(M1)]
\item There are $0 < \alpha < \beta$ and $\epsilon >0$ such that
$$ m(\alpha)=m(\beta)=1, \qquad \E \abs{Q}^{\beta+\epsilon} < \infty, \qquad { \E \norm{A_1^*}^{\beta+\epsilon} \iota(A_1^*)^{-\epsilon}< \infty,}
$$
\item there is a nondegenerate random variable $X$ satisfying \eqref{eq:SFPE} with $\E \abs{X}^s<\infty$ for all $s < \beta$.
\end{enumerate}
\end{enumerate}
Then for this $X$   and all $u \in \S$,
$$ \liminf_{t \to \infty} t^\beta \P{\skalar{u,X}>t} > 0.$$
\end{thm}

As a corollary of this results we obtain that the asymptotic behavior proved in \cite[Theorems 2.7, 2.9 and 2.11]{BDMM2013} and \cite[Theorem 1.9]{Mirek2013} is exact. Tail estimates for the case of random $N$ in the multivariate situation have not yet been considered in the literature and this is the first result in that direction.

\subsection{Structure of the paper}
We proceed in Section \ref{sect:wbp} by introducing the weighted branching process, which allows for the study of the fixed point equation \eqref{eq:SFPE} by iteration and for the  construction of random variables, which satisfy the equation a.s.~(in contrast to in law).
Using that the support of these random variables is unbounded, we can estimate $\P{\abs{X}>t}$ from below by a union of events of the type ``one large term occurs'', this is made precise in Section \ref{sect:estimate}, with the fundamental estimate being proved in Lemma \ref{lem:estimate}. Section \ref{sec: combinatorics} is mainly combinatorical, there we count the number of events occuring in the union, and estimate from above the probability of intersections, which we make small by an appropriate choice of parameters and thereby complete the proof of the main theorem in Section \ref{subsect:proofmthm}. An outline of the proof is given in Subsection \ref{sect:outline}.

\begin{rem}
We have tried hard, but were not able to avoid the case distinctions concerning $N$. A natural way to do this would be the use of a spinal-tree-identity (many-to-one lemma), but it seems that our approach is not compatible with this technique. The main difficulty is that we consider sums over particular subtrees (as defined in \eqref{www}), which we were not able to reformulate in such a way that a many-to-one-lemma would be applicable.
\end{rem}

\section{Weighted branching process}\label{sect:wbp}

In this section, we introduce the weighted branching process, i.e.~a sequence of random variables which satisfy Eq. \eqref{eq:SFPE} almost surely.

\subsection{Trees}
Let $\N=\{1,2,\ldots\}$ be the set of positive integers and let
$$
\bU = \bigcup_{k=0}^\8 \N^k
$$
be the set of all finite sequences $\bi = i_1\ldots i_n$. By $\emptyset$ we denote the empty sequence. For $\bi = i_1\ldots i_n$ we denote by $|\bi|$ its length and by $\bi|_k = i_1\ldots i_k$ the curtailment of $\bi$ up to first $k$ terms. Given $\bi\in\bU$ and $j\in\N$ we define $\bi j = i_1\ldots i_n j$ the sequence obtained by juxtaposition. In the same way we define $\bi\bj$ for $\bi, \bj\in \bU$.

We introduce a partial ordering on $\bU$, writing $\bi \le \bj$ when there exists $\bi_1\in \bU$ such that $\bj = \bi\bi_1$. If $\bi,\bi'\in\bU$, we write
$\bj = \bi\wedge \bi'$ for the maximal common sequence of $\bi$ and $\bi'$, that is, $\bj$ is the longest sequence such that $\bj\le \bi$ and $\bj \le \bi'$.

We say that a subset $\T$ of $\bU$ is a tree if
\begin{itemize}
\item $\emptyset \in \T$;
\item if $\bi\in\T$, then $\bi|_k\in\T$ for any $k< |\bi|$;
\item if $\bi\in \T$ and $j\in \N_+$ then $\bi j\in T$ if and only if $1\le j \le N_{\bi}$, for some integer $N_{\bi}\ge 0$.
\end{itemize}
Then $\emptyset$ is the root of the tree.

{ In case (N-random),  let $(N_{\bi})_{\bi \in \bU}$ be a family of i.i.d.~copies of $N$, which thus determines the shape of the tree $\T$. By $\F_\T$ we will denote below the $\sigma$-algebra generated by $(N_{\bi})_{\bi \in \bU}$.  In case (N-fixed), the shape of the tree is deterministic, then $\T=\bigcup_{k=0}^\8 \{1,\dots,N\}^k$.}

\subsection{Random variables indexed by $\U$}
%

To each node $\bj \in \U$ we attach an independent copy $\A_\bj:=(Q_\bj, (A_{\bj i})_{i \ge 1})$ of $\A:=(Q, (A_i)_{i \ge 1})$ and, given a random variable $X \in \Rd$, satisfying \eqref{eq:SFPE}, an independent copy $X_\bj$ of $X$ as well. We identify $(Q_{\emptyset}, (A_{\emptyset i})_{i \ge 1}) = (Q, (A_i)_{i \ge 1})$. We refer to $A_{\bj i}$ as the weight pertaining to the edge connecting $\bj$. Denote the total weight on the unique path connecting the edge $\bj$ with the edge $\bj \bi$ by
$$  \Pi_{\bj, \bj \bi} ~:=~ A_{\bj i_1} A_{\bj i_1 i_2} \cdots A_{\bj \bi}, \qquad \Pi_\bj := \Pi_{\emptyset, \bj}$$
and define the empty product to be the $d \times d$ identity matrix.
 Due to the assumption $N < \infty$ $\Pfs$, each generation of $\T$ has a.s.~a finite size. Notice also that in view of \eqref{eq:perm}
 the law of $\Pi_{\bj,\bj\bi}$ depends only on the numbers of factors and coincides with the law of $\Pi_{\bi}$.
 \medskip

Recall that we defined $\mu$ to be the law of $\mA_1^*$ and $\mM_1, \mM_2, \dots$ to be a sequence of i.i.d. random variables with law $\mu$. Then $\Pi_n^* := \mM_n \cdots \mM_1$ has the same law as $\Pi_\bi^*$ for every $\bi \in \T$ with $\abs{\bi}=n$ and moreover,
$$ \P{\Pi_\bi^* \as u \in \cdot} = \P[u]{U_n \in \cdot}, \qquad \P{\log \abs{\Pi_\bi^* u} \in \cdot} = \P[u]{S_n \in \cdot}$$ (for the definition of $(U_n, S_n)_{n \ge 0}$ see \eqref{unsn}).
\medskip

We write $[\T]_{\bj}~:=~\{ \bi \in \bU \, : \, \bj \bi \in \T \}$ for the subtree of $\T$ rooted at $\bj$, and define in general the shift operator acting on functions of the family $(\A_\bi, X_\bi)_{\bi \in \U}$ by
$$ \left[F\big((\A_\bi, X_\bi)_{\bi \in \U}\big) \right]_\bj ~:=~ F\big((\A_{\bj \bi}, X_{\bj \bi})_{\bi \in \U}\big).$$
With this notation, $\Pi_{\bj, \bj \bi} = [\Pi_\bi]_\bj.$ The random variables \begin{equation}\label{eq:WBP}
\begin{split}
Y_l &:=~ \sum_{\abs{\bi} < l} \Pi_{\bi} Q_{\bi} + \sum_{\abs{\bi}=l} \Pi_{\bi} X_{\bi},\qquad l\ge 1,\\
Y_0 &:= X_{\emptyset}.
\end{split}
\end{equation}
are called {\em the weighted branching process associated with $(Q, (A_i)_{i \ge 1})$ and $X$}. They  satisfy
$$ Y_l ~=~ \sum_{i = 1 }^N A_i [Y_{l-1}]_i + Q,$$
where $[Y_{l-1}]_i$ are i.i.d., with the same law as $Y_{l-1}$.
Since  $X_{\bi}$ are solutions to \eqref{eq:SFPE}, then in particular, $Y_l \eqdist X$ for all $l \in \N$.
%

We define moreover
\begin{eqnarray} \label{n4.2}
Z_{l,\bi k} &:=& \sum_{j\not=k, j\le N_{\bi}} A_{\bi j} [Y_{l-\abs{\bi}-1}]_{\bi j} + Q_{\bi}, \qquad l>|\bi|.
\end{eqnarray}
Then for $l=|\bi|+1$ we have
\begin{eqnarray} \label{n4.3}
Z_{l,\bi k} &:=& \sum_{j\not=k, j\le N_{\bi}} A_{\bi j} X_{\bi j} + Q_{\bi}.
\end{eqnarray}
By \eqref{eq:perm} the random variables $(Z_{l, \bi k})_{1 \le k \le N_{\bi}}$ are obviously identically distributed. To simplify our notation we define
$$
Z_{|\bi|} := Z_{|\bi|,\bi}
$$


For every $l \in \N$ and $\bi \in \T$ with $\abs{\bi} \le l$, we can rearrange the sum in Eq. \eqref{eq:WBP} to obtain the a.s. identity.
\begin{equation}
\label{n4.4}
Y_l ~=~ \Pi_{\bi} [Y_{l-|\bi|}]_{\bi} + \sum_{k \le \abs{\bi}} \Pi_{\bi|_{k-1}}Z_{l,\bi|_{k}}.
\end{equation}
Observe that this implies for $|\bi|\le l$ the following identity in law.
\begin{equation}
\label{n4.4'}
Y_l ~\eqdist~ \Pi_{\bi} X_{\bi} + \sum_{k \le \abs{\bi}} \Pi_{\bi|_{k-1}}Z_{l,\bi|_{k}}.
\end{equation}

\subsection{Outline of the proof}\label{sect:outline}
This identity may give a first idea, how we are going to proceed in the proof of the main theorems: We consider sets where $\Pi_{\bi}X_\bi$ is large, while the remaining sum is small. Therefore, we in turn study sets where $\norm{\Pi_\bi}$ is large, but smaller products are comparably small (with the comparison governed by a parameter $C_0$). The probability of such sets will be estimated using large deviation results for products of random matrices, obtained in \cite{BM2013}. Then the probability that $X$ is large will be estimated from below by the union of sets as described above, over different $\bi$. It will be convenient to not take the union over $\bi$ from the whole tree, but rather from a sparse subtree, in order to make the events sufficiently disjoint. The relative size of the subtree will be given by a parameter $C_1$, which will be a free parameter of the proof.

\medskip
A particular problem in the multivariate situation is to compare $\Pi_\bi X_\bi$ with $\norm{\Pi_\bi}$. We deal with this question at the beginning of the next section, the better part of which is devoted to formulate precisely the heuristics we described above.

\section{First estimates}\label{sect:estimate}
We start this section by a lemma stating that $X$ has unbounded support in ``all'' directions of $\Rd$ resp. $\Rdnn$, which we will make use of subsequently in Lemma \ref{lem:estimate}, which gives the fundamental comparison between $\P{\abs{X}>t}$ and the union of large deviation events.

\begin{lem}\label{lem:cones}
Assume that the hypotheses of Theorem \ref{thm:d>1} are satisfied and that $X$ is not a.s.~constant.
 Then for all $D >0$ there is $J < \infty$ and $\epsilon_j>0$, $1 \le j \le J$, and a $\kappa>0$,  such that there are disjoint subsets $\Omega_j$ of $\Rd$ 
with
$$\P{ X\in \Omega_j \ \text{ and } |X|> \frac{D}{\epsilon_j}} \ge \kappa$$
and moreover
\begin{equation}
\label{eq: 4.2cones}
 \Rd \subset \bigcup_{j=1}^J \Omega_j^*,
 \end{equation}
where $\Omega_j^*$ are the cones
$$ \Omega_j^* := \{ z \in \Rd \, : \, \skalar{z,x} \ge \epsilon_j \abs{z} \abs{x} \text{ for all $x \in \Omega_j$}\}. $$
If $\mu$ satisfies (C), then the same statement is valid, but with $\Omega_j$ being subsets of $\Rdnn$ and \eqref{eq: 4.2cones} replaced by
$$ \R_{\ge} \subset \bigcup_{j=1}^J \Omega_j^*,
$$
\end{lem}

\begin{proof}

Let $X_1,\dots, X_N$ be i.i.d.~copies of $X$ (with $N$ constant or random). Set $B := \sum_{i=2}^N \mA_i X_i + Q$. Since $X_i$ are i.i.d., nontrivial and independent of $(A_i)_{i \ge 1}$ and $Q$, it follows that also $B$ must be nontrivial. Moreover, due to the moment assumptions (M1) and (M2) and the convexity of $k(s)$, there is $s \in (\alpha, \beta)$ s.t.~$k'(s)>0$ and $B$ has a finite moment of order $s$. Then $X$ satisfies the equation $X \eqdist \mA_1 X_1+ B$, and for $X$ satisfying such an equation, the results are shown in \cite[Lemma 10.2]{BM2013}. Note that there only the condition $k'(s) >0$ is relevant (which follows from the convexity of $k$); the additional condition (stated there), that $k(s)=1$ is not needed.

%
%
\end{proof}

Now we turn to the announced estimate from below for $\P{\skalar{u,X}>t}$.
Our estimates will be given in terms of the sets
\begin{align*}V_{\bi,t} ~:=&~ \Big\{ |\Pi_{\bi}^* u|\ge t\ \mbox{ and } \big\|{\Pi^*_{{\bi}|_k}} \big\| { (\abs{Z_{\bi|_{k+1}}}\vee 1)}\le e^{-(|{\bi}|-k)\delta} C_0 t \ \forall k< |{\bi}|     \Big\},\\
W_{\bi,\bi',t} ~:=&~ \Big\{{\abs{\Pi_{\bi}^*u} > t, \ \abs{\Pi_{\bi'}^*u}> t, \ \norm{\Pi_{\bi\wedge \bi'}}  \le C_0 t e^{-\delta({\abs{\bi}-|\bi\wedge \bi'|})}}\Big\},
\end{align*}
for some constants $C_0$ and $\delta$ that will be defined below.



\begin{lem}\label{lem:estimate}
For all $u \in \S$, $C_0 >0$ there is $\kappa >0$ such that for all $t >0$ and all a.s. subsets $\W \subset \T$,
\begin{equation}\label{eq:estimate}
\P{\skalar{u,X} > t} ~\ge~  \kappa \E\bigg[ \sum_{{\bi}\in\W} \P{ V_{{\bi,t}}}  \bigg] -
 \E\bigg[ \sum_{{\bi}\in\W}\sum_{\bi' \in \W, \abs{\bi'} \le \abs{\bi}, \bi \neq \bi'} \P{W_{\bi, \bi',t}}  \bigg] \in [-\infty,\infty).
\end{equation}
\end{lem}

\begin{proof}

 Fix $l \in \N$. From Eq. \eqref{n4.4} we obtain that
 \begin{align} \label{eq:uYl}
 \abs{\skalar{u,Y_l}} \ge \abs{\skalar{\Pi_\bi^* u, [Y_{l-\abs{\bi}}]_\bi}} - \sum_{k \le \abs{\bi}} \norm{\Pi_{\bi|_{k-1}}} \abs{Z_{l, \bi|_k}}.
 \end{align}

 For arbitrary $C_0, \delta, t>0$ let $D = 1+C_0/(1-e^{-\delta})$ and  introduce the family of sets
$$
\wt V_{ \bi,t} := \bigcup_{j =1}^J \bigg(V_{\bi,t} \cap \Big\{ X_{\bi}\in \Omega_j \ \mbox{ and } |X_{\bi}|>\frac{D}{\epsilon_j}  \Big\} \cap
\big\{ \Pi_{\bi}^*u\in \Omega_j^*  \big\}  \bigg),
$$
where the cones $\Omega_j$ and $\Omega_j^*$ were defined in Lemma \ref{lem:cones},
as well as the sets
\begin{align*} \wt V_{l,\bi,t} ~:=~ \bigcup_{j=1}^J & \Big( \big\{ \abs{\Pi_\bi^* u} \ge t \text{ and } \norm{\Pi_{\bi|_k}^*} ( \abs{Z_{l,\bi|_{k+1}}} \vee 1) \le e^{-(\abs{\bi}-k)\delta}C_0 t \, \forall k < \abs{\bi} \big\} \\
& \cap  \big\{ [Y_{l-\abs{\bi}}]_\bi \in \Omega_j \text{ and } \abs{ [Y_{l-\abs{\bi}}]_\bi} > \frac{D}{\epsilon_j} \big\} \, \cap \, \big\{ \Pi_\bi^*u \in \Omega_j^* \big\} \Big).
 \end{align*}


The latter sets are defined in such a way, that on $\wt V_{l,\bi,t}$
\begin{align*}
 \abs{\skalar{u,Y_l}} ~\ge&~ \epsilon_j \abs{\Pi_\bi^* u} \abs{[Y_{l-\abs{\bi}}]_\bi} - \sum_{k \le \abs{\bi}}  \norm{\Pi_{\bi|_{k-1}}} \left( \abs{Z_{l, \bi|_k}} \vee 1 \right) \\
 \ge&~ Dt - \sum_{k\le \abs{\bi}} e^{-(\abs{\bi}-k-1)\delta}C_0 t,
 \end{align*}
 which is larger than $t$ upon choosing $D$ large enough. Here again, we need the a.s.~version \eqref{n4.4}.

 Since $X \eqdist Y_l$ for any $l \in \N$, we obtain the following estimate
 \begin{align*}
 \P{\abs{\skalar{u,X}}>t} ~\ge&~ \P{ \bigcup_{\abs{\bi} \le l} \wt V_{l,\bi,t}} ~\ge~ \P{ \bigcup_{\abs{\bi} \le l, \bi \in \W} \wt V_{l, \bi, t}}.
 \end{align*}
 which is valid for any a.s. subset $\W \subset \T$. Below, we use the shorthand $\W_l = \{ \bi \in \T \, : \, \abs{\bi} \le l, \bi \in \W\}.$

Assuming that $\sup_{l \in \N} \E \sum_{\bi \in \W_l} \1[{\wt V_{l,\bi,t}}] < \infty$ (this will be shown below), we use the inclusion-exclusion formula and separate as follows:
 \begin{align*}
 \P{ \bigcup_{\abs{\bi} \le l, \bi \in \W} \wt V_{l, \bi, t}}
 ~\ge&~\E \left( \sum_{\bi \in \W_l} \P{ \wt V_{l,\bi,t} \, | \, \F_{\bi}} \right) - \E \left( \sum_{\bi,\bi' \in \W_l; \bi \neq \bi', \abs{\bi'} \le \abs{\bi}} \1[{\wt V_{l,\bi,t} \cap \wt V_{l,\bi',t}}] \right).
 \end{align*}
Now observe that $\wt V_{l,\bi,t} \cap \wt V_{l,\bi',t} \subset W_{\bi,\bi',t}$, which gives the estimate
 \begin{align*}
 \P{ \bigcup_{\abs{\bi} \le l, \bi \in \W} \wt V_{l, \bi, t}}
 ~\ge&~\E \left( \sum_{\bi \in \W_l} \P{ \wt V_{l,\bi,t} \, | \, \F_{\bi}} \right) - \E \left( \sum_{\bi \in \W_l} \, \sum_{\bi' \in \W_l, \bi \neq \bi', \abs{\bi'} \le \abs{\bi}} \1[{W_{\bi,\bi',t}}] \right).
 \end{align*}
As for obtaining Eq. \eqref{n4.4'} from Eq. \eqref{n4.4}, we have that
$$ \P{ \wt V_{l,\bi,t} \, | \, \F_{\bi}} ~=~ \P{\wt V_{\bi,t}} \quad \Pfs.$$
The union in the definition of $\wt V_{\bi,t}$ is over disjoint sets, since the $\Omega_j$ are disjoint. Furthermore, $X_{\bi}$ is independent of $V_{\bi,t}$ and $\Pi^*_\bi$, and has the same law as $X$, hence
\begin{align}\label{eq:vit}
\P{\wt V_{\bi,t}} ~=&~ \sum_{j=1}^J \P{V_{\bi,t} \cap \{ \Pi_\bi^*u \in \Omega_j^*\}} \P{X_\bi \in \Omega_j \text{ and } \abs{X_\bi} > \frac{D}{\epsilon_j}} \\
\notag =&~  \P{X \in \Omega_j \text{ and } \abs{X} > \frac{D}{\epsilon_j}} \sum_{j=1}^J \P{V_{\bi,t} \cap \{ \Pi_\bi^*u \in \Omega_j^*\}}  \\
\notag     \ge&~ \kappa \P{V_{\bi,t}}.
\end{align}
Thus, we have obtained the following estimate, valid for all $l \in \N$:
\begin{equation}
\P{\abs{\skalar{u,X}}>t} ~\ge~ \kappa \E \left( \sum_{\bi \in \W_l} \P{ V_{\bi,t}} \right) - \E \left( \sum_{\bi \in \W_l} \, \sum_{\bi' \in \W_l, \bi \neq \bi', \abs{\bi'} \le \abs{\bi}} \1[{W_{\bi,\bi',t}}] \right).
\end{equation}

We finally have to justify that $\sup_{l \in \N} \E \sum_{\bi \in \W_l} \1[{\wt V_{l,\bi,t}}] < \infty$. But estimating similar as in \eqref{eq:vit}, we obtain that
$\P{\wt V_{\bi,t}} \le J \P{V_{\bi,t}}$; and the supremum is obviously bounded by
$$J \E \sum_{\bi \in \T} \P{V_{\bi,t}} ~\le~ J \E \sum_{\bi \in \T} \P{\abs{\Pi_\bi^*u}>t} ~=~ J C \sum_{n=0}^\infty (\E N)^n \Prob_u(S_n > \log t) ~\le~  C J \sum_{n=0}^\infty \frac{c_s m(s)^n}{{ t^{s}}} ~<~\infty,$$
where we used the Markov inequality for some $s \in (\alpha, \beta)$ and the estimate \eqref{eq:kEPi} in the last step.
\end{proof}

\subsection{Estimates for $\P{V_{\bi,t}}$, $\P{W_{\bi,\bi',t}}$}\label{sect:ViWii}

The further analysis of Eq. \eqref{eq:estimate} splits in two parts. On the one hand, we have to estimate the probabilities appearing there, in terms of the distance between $\bi$ and $\bi'$ and on the other hand, we have to do some combinatorics on the tree, in order to do the summation. In this section, we bound the probabilities. The set $\W$ will be defined precisely in the next section. However it will  be a subset of the tree $\T$ consisting of vertices $\bi$ such that
$$
n_t - \sqrt{n_t} < |\bi| < n_t - \sqrt{n_t}/2,
$$ where $n_t = \lceil  \log t/\rho \rceil  $ and $\rho = m'(\beta)$. Thus the estimates provided below will be only for this particular set of indices.

\subsection{Probability of $V_{\bi,t}$}
%
%
%
In view of \eqref{eq:perm}, the probability of $V_{\bi,t}$ does in fact only depend on $n=\abs{\bi}$. Let $(M_n, Z_n)$ be a sequence of i.i.d.~copies of $(A_1^*, Z)$ for $Z = \sum_{i=2}^N A_i X_i + Q$. Define, with $\Pi_n^*:= M_n \cdots M_1$,
$$ V_{n,t} ~:=~ \Big\{ |\Pi_n^* u|\ge t\ \mbox{ and } \norm{\Pi^*_{k}} { (\abs{Z_{{k+1}}}\vee 1)}\le e^{-(n-k)\delta} C_0 t \ \forall k< n     \Big\} .$$
Then $\P{V_{\bi,t}}=\P{V_{n,t}}$ as soon as $\abs{\bi}=n$.

The sets $V_{n,t}$ were already considered for $d=1$ in \cite{BDZ2014+} (Theorem 2.3) and for $d\ge 2$ (under the same hypotheses as in the present paper) in \cite{BM2013} (Lemma 10.5). We refer to these two papers for the proofs of the following lemmas.

\begin{lem} \label{lem:estimate Vn} Assume  that $\E|Z|^{\gamma}<\8$ for some $\gamma >0$, then
there are constants $\delta, C_0,D_1,D_2,N_0 >0$ such that
$$
D_1 \cdot \frac{k(\beta)^n}{\sqrt{n_t} e^{\beta n_t \rho}}
\le \P{\abs{\Pi_n^* u}>t} \le
D_2 \cdot \frac{k(\beta)^n}{\sqrt{n_t} e^{\beta n_t \rho}},
$$
$$
D_1 \cdot \frac{k(\beta)^n}{\sqrt{n_t} e^{\beta n_t \rho}}
\le \P{V_{n,t}} \le
D_2 \cdot \frac{k(\beta)^n}{\sqrt{n_t} e^{\beta n_t \rho}}.
$$ for all $\lceil \log t / \rho \rceil = n_t > N_0$ and every $n_t-\sqrt{n_t} \le n \le n_t - \sqrt{n_t}/2$.

For the assertion of this lemma to hold, $k(\beta) =1$ is not necessary, we only need that $k'(\beta)>0$.
\end{lem}

\subsection{Probability of $W_{\bi,\bi',t}$}

The main result of this subsection is the following lemma
\begin{lem}\label{lem64}
Assume  that the hypotheses of Theorem \ref{thm:d>1} are satisfied. Then there is $\chi >0$ such that for all $t > e^{\rho N_0}$ and all $|\bi|\ge |\bi'|$ with $n_t - \sqrt{n_t} \le \abs{\bi}, \abs{\bi'} \le n_t - \sqrt{n_t}/2$ it holds that
\begin{equation}\label{eq:estimateWii} \P{W_{\bi,\bi',t}} \le \frac{C_2 k(\beta)^{\abs{\bi}}}{t^\beta \sqrt{n_t}} \, \frac{k(\beta)^{\abs{\bi'}-\abs{\bi \wedge \bi'}}}{e^{\chi(\abs{\bi}-\abs{\bi \wedge \bi'})}}, \end{equation}
with a constant $C_2$ which is independent of $t$.
\end{lem}
\begin{proof}
We will start with some general calculations, and then deal with the cases (N-random) and (N-fixed) separately.
Denote the joint law of $(A_1, A_2)$ by $\eta$. Recall that  for each $\bi \in \T$, $1 \le k < l \le N_{\bi}$, $\law{(A_{\bi k}, A_{\bi l})} = \eta$ as well.
For $\bi, \bi' \in \T$, we   write
$$\bi_0= \bi \wedge \bi', \quad m = \abs{\bi \wedge \bi'} = \abs{\bi_0}, \quad \bi|_{m+1}=\bi_0k, \ \bi'|_{m+1}=\bi_0l, \quad p=\abs{\bi}, \ q=\abs{\bi'}, \quad U_\bi := \Pi_\bi^* u/ |\Pi_\bi^* u|$$ and recall the notation
$ \Pi_{\bi, \bi\bj}= A_{\bi j_1} \cdots A_{\bi \bj}$, $ \Pi^*_{\bi, \bi\bj}= A^*_{\bi \bj} \cdots A^*_{\bi j_1}$
for the product of the weights along the path between $\bi$ and $\bi \bj$.
Then
\begin{align}
\P{W_{\bi,\bi',t}} ~=&~ \P{\abs{\Pi_{\bi}^*u} > t, \ \abs{\Pi_{\bi'}^*u}> t, \ \norm{\Pi_{\bi\wedge \bi'}}  \le C_0 t e^{-\delta({\abs{\bi}-m})}} \nonumber \\
\le&~ \P{\abs{\Pi_{\bi_0k,\bi}^*U_{\bi_0k}} \abs{A_{\bi_0k}^*U_{\bi_0}} \abs{\Pi_{\bi_0}^* u} > t, \ \norm{\Pi_{\bi_0l,\bi'}^*} \norm{A_{\bi_0l}} \norm{\Pi_{\bi_0}^*} > t, \ \norm{\Pi_{\bi_0}^*}  \le C_0 t e^{-\delta({\abs{\bi}-m})}} \nonumber\\
\le&~ \P{\abs{\Pi_{\bi_0k,\bi}^*(A_{\bi_0k}^* \as U_{\bi_0})} \abs{A_{\bi_0k}^*U_{\bi_0}} \abs{\Pi_{\bi_0}^* u} > t, \ \norm{\Pi_{\bi_0l,\bi'}^*} > \frac{e^{\delta({\abs{\bi}-m})}}{C_0 \norm{A_{\bi_0l}}}   } \nonumber\\
=&~ \int \, \P{\abs{\Pi_{\bi_0k,\bi}^* (a_1^* \as U_{\bi_0})} \abs{a_1^*U_{\bi_0}} \abs{\Pi_{\bi_0}^* u} > t} \,  \P{\norm{\Pi_{\bi_0l,\bi'}^*} > \frac{e^{\delta({\abs{\bi}-m})}}{C_0 \norm{a_2}}   } \eta(da_1, da_2)\nonumber\\
=&~ \int \, \P{\abs{\Pi_{m+1,p}^* (a_1^* \as U_{m})} \abs{a_1^*U_{m}} \abs{\Pi_m^* u} > t} \,  \P{\norm{\Pi_{m+1,q}^*} > \frac{e^{\delta({p-m})}}{C_0 \norm{a_2}}   } \eta(da_1, da_2)\nonumber\\
\le&~\int \, \P{\abs{\Pi_{m+1,p}^* (a_1^* \as U_{m})} \abs{a_1^*U_{m}} \abs{\Pi_m^* u} > t} \ \frac{\E[\norm{\Pi_{q-m-1}}^\alpha] \, C_0^\alpha \norm{a_2}^\alpha}{e^{\alpha\delta({p-m})}}   \, \eta(da_1, da_2) \label{eq:Wii1},
%
%
\end{align}
where we conditioned upon $(A_{\bi_0k},A_{\bi_0l})$ and used the Markov inequality in the last step. The reason that we applied it with the exponent $\alpha$ is that we will be able to replace $\E \norm{\Pi_n}^\alpha$ by $k(\alpha)^n$, but the latter one equals also $k(\beta)$.

>From now on, we will consider the cases (N-random) and (N-fixed) separately.

\medskip

{\sc Case} (N-random)

In this case, $a_1$ and $a_2$ are independent, and Eq. \eqref{eq:Wii1} simplifies to
\begin{align*}
\P{W_{\bi,\bi',t}} ~\le &~ \P{\abs{\Pi_{p}^*u} > t} \ \frac{\E[\norm{\Pi_{q-m-1}}^\alpha] \, C_0^\alpha \E[\norm{M}^\alpha]}{e^{\alpha\delta({p-m})}} \\
\le&~  \frac{D_2' \,k(\beta)^p}{\sqrt{n_t}\, t^{\beta }} \ \frac{c_\alpha^2 k(\beta)^{q-m} \, C_0^\alpha }{e^{\alpha\delta({p-m})}} \\
=&~  \frac{C_2 \,k(\beta)^{\abs{\bi}}}{\sqrt{n_t}\, t^{\beta }} \ \frac{k(\beta)^{\abs{\bi'}-\abs{\bi \wedge \bi'}}  }{e^{\chi({\abs{\bi}-m})}}
\end{align*}
by Lemma \ref{lem:estimate Vn} and Eq. \eqref{eq:kEPi} (recall $k(\alpha)=(\E N)^{-1}=k(\beta)$), with $C_2 := D_2' c_\alpha^2 C_0^\alpha$ and $\chi := \alpha \delta$.

\medskip

{\sc Case} (N-fixed).

In this case, $\norm{a_2}$ is bounded by some constant $c_A$, say,  and Eq. \eqref{eq:Wii1} simplifies to
\begin{align*}
\P{W_{\bi,\bi',t}} ~\le &~ \int \, \P{\abs{\Pi_{m+1,p}^* (a_1^* \as U_{m})} \abs{a_1^*U_{m}} \abs{\Pi_m^* u} > t} \ \frac{\E[\norm{\Pi_{q-m-1}}^\alpha] \, C_0^\alpha c_A^\alpha}{e^{\alpha\delta({p-m})}}   \, \eta(da_1, da_2) \\
=&~ \P{\abs{\Pi_{p}^*u} > t} \ \frac{\E[\norm{\Pi_{q-m-1}}^\alpha] \, C_0^\alpha c_A^\alpha}{e^{\alpha\delta({p-m})}},\end{align*}
and from here, we can proceed as before to obtain the estimate \eqref{eq:estimateWii} with
 $C_2 := D_2' c_\alpha^2 C_0^\alpha c_A^\alpha / k(\beta)$ and $\chi := \alpha \delta$.
\end{proof}

\section{Combinatorics on the tree}
\label{sec: combinatorics}

In this section we consider $N$ to be random. If $N$ were fixed, then the calculations are similar but easier.
As we mentioned above the subset $\W$ of $\T$ will contain only some of the nodes satisfying
$n_t - \sqrt{n_t} \le \abs{\bi} \le n_t - \sqrt{n_t}/2$
and therefore will also depend on $t$. We will consider only a sparse subset of those nodes.
Namely only nodes from every $C_1$th-generation, which moreover end with $C_1$ one's. The number $C_1 \in \N$ will be a parameter of the proof, to be fixed at the very end. It's choice will be independent of $t$. {  Note however, that the estimate $\P{\skalar{u,X}>t} \ge \epsilon t^\beta$ is only valid for large enough $t$, namely $t > C_1e^{N_0 \rho}$.  }{ Since it is sufficient to show that $\liminf_{t \to \infty} t^{\beta}\P{\skalar{u,X}>t} \ge \epsilon >0$, we will even restrict to such $t$, for which $\sqrt{n_t}/2C_1$ is an integer. This is not really necessary, but simplifies expressions.}

Below, we will often use that $1=m(\beta)=k(\beta) \E N$, and therefore, $k(\beta) = (\E N)^{-1}$. Keep in mind, that under (N-random), the shape of the tree is random, as will be that of its subset. Therefore, we have to take expectations with respect to the shape of the tree.

\medskip

To be precise, upon fixing $t$, let $n_t = \lceil \log t/\rho \rceil$. We will consider nodes the generations of which are from the set
$$
L_t=\{ k \in C_1\N : n_t - \sqrt{n_t} \le kC_1 < n_t - \sqrt{n_t}/2 \}.
$$ { Note $\# L_t = \frac{\sqrt{n_t}}{2C_1}$ since we assume the latter to be integer.}  Denote by  $\bo_{C_1}=1\ldots 1\in \bU$ the sequence consisting of $C_1$ one's.
%
Define
\begin{equation}\label{www}
\W = \Big\{ \bi\in \T:\; |\bi|\in L_t \mbox{ and }\ \bi = \bi|_{|\bi|-C_1}\bo_{C_1}
\Big\}.\end{equation}

We will calculate below several times the expected number of elements of $\W$ lying on the level $k \in L_t$:
\begin{equation}
\label{eq: pos:5}
\E\big[ \#\{\bi\in \W:\; |\bi|=k \}\big] = \E\big[ \#\{\bi\in  \T:\; |\bi|=k -C_1 \}\big] = \big(\E N\big)^{k - C_1} = k(\beta)^{C_1 - k}
\end{equation}

\begin{lem}\label{lem:sumVi}
For all $t$ large enough, { (and such that $\sqrt{n_t}/2C_1 \in \N$)}
$$ \E\bigg[ \sum_{{\bi}\in\W} \P{ V_{{\bi,t}}}  \bigg] ~\ge~ \frac{D_1 k(\beta)^{C_1}}{2C_1} \cdot \frac 1{t^{\beta}}.$$
\end{lem}

\begin{proof} Using Lemma \ref{lem:estimate Vn},
\begin{align*}
\E\bigg[ \sum_{{\bi}\in\W} \P{V_{{\bi,t}}}  \bigg] ~\ge&~ \frac{D_1}{t^{\beta}}\cdot \frac 1{\sqrt{n_t}}  \, \E\bigg[ \sum_{{\bi}\in\W} k(\beta)^{|{\bi}|} \bigg]\\
  ~\ge&~ \frac{ D_1}{t^{\beta}}\cdot \frac 1{\sqrt{n_t}} \, \sum_{l \in L_t} k(\beta)^{l}  \E\big[\#\{ {\bi}\in \W:\; |{\bi}|=l\} \big]\\
~=&~ \frac{D_1 k(\beta)^{C_1}}{2C_1} \cdot \frac 1{t^{\beta}}. \qedhere
\end{align*}
\end{proof}

This in particular shows the finiteness of the left-hand-side for each $t$ and this particular subset $\W$. Therefore, Lemma \eqref{lem:estimate} applies to give the estimate of the first term in \eqref{eq:estimate}, and we can proceed computing the sum over the mixed terms.

\subsection{Calculations involving $W_{\bi,\bi',t}$}

\begin{lem}\label{lem:sumWii}
There is $0 < \eta < \chi$ (independent of $t$), such that for all $t$ large enough
$$\E\bigg[ \sum_{{\bi}\in\W}\sum_{\bi' \in \W, \abs{\bi'} \le \abs{\bi}, \bi \neq \bi'} \P{W_{\bi, \bi',t}}  \bigg]  ~\le~ \frac{C_2 k(\beta)^{2C_1}}{e^{\eta C_1}} \frac{1}{t^\beta}.$$
\end{lem}

\begin{proof}
\Step[1]: We are going to reorganize the summation over $\bi, \bi' \in \W$, by ordering them according to latest common ancestor, $\bi_0 :=\bi \wedge \bi'$.
Introducing as before the notation
$$ \W_{\bi_0,t} = \{ \bi\in  \W:\; \bi\ge \bi_0 \}, \qquad   \abs{\bi \wedge \bi'} ~=~ m, \qquad \abs{\bi}~=~p, \qquad \abs{\bi'} ~=~q,$$
the restrictions $\bi_0 = \bi \wedge \bi'$, $\bi, \bi' \in \W, \abs{\bi'} \le \abs{\bi}$ translate to (we omit the restriction $\bi \neq \bi'$, since anyway we want an estimate from above)
$$ \bi, \bi' \in \W_{\bi_0,t}, \qquad \max\{m + C_1, n_t - \sqrt{n_t} \} \le p \le n_t - \sqrt{n_t}/2, \qquad m \le q \le p.$$
Similar to Eq. \eqref{eq: pos:5}, we compute for $l \ge m$, $l \in L_t$, the expected size of the $l'th$ generation in $\T_{\bi_0,t}$ to be
$$ \E\big[ \#\{\bi\in \W_{\bi_0,t}:\; |\bi|=l \}\big] = \E\big[ \#\{\bi\in  \T_t:\; |\bi|=l -C_1, \, \bi|_{m}=\bi_0 \}\big] = \big(\E N\big)^{l - m - C_1} = k(\beta)^{C_1 + m - l}.
 $$

\newcommand{\pt}{p_t^*}
\newcommand{\nt}{n_t^*}
Abbreviate the lower bound for $p$ by $\pt := \max\{m + C_1, n_t - \sqrt{n_t} \}$, and the upper bound $\nt := n_t - \sqrt{n_t}/2$.
Then, using \eqref{eq:estimateWii}
\begin{align*}
&~\E\bigg[ \sum_{{\bi}\in\W}\sum_{\bi' \in \W, \abs{\bi'} \le \abs{\bi}, \bi \neq \bi'} \P{W_{\bi, \bi',t}}  \bigg]  \\
~\le&~ \E \bigg[    \sum_{m \le \nt} \, \sum_{\pt \le p \le \nt } \, \sum_{\bi \in \T_{\bi_0,t}, \abs{\bi}=p} \, \sum_{m \le q \le p} \, \sum_{\{\bi':\; \bi_0 = \bi\vee\bi' , \abs{\bi'}=q\} } \, \frac{C_2 k(\beta)^{p}}{t^\beta \sqrt{n_t}} \, \frac{k(\beta)^{q-m}}{e^{\chi(p-m)}} \bigg] \\
~\le&~ \E \bigg[    \sum_{m \le \nt} \, \sum_{\pt \le p \le \nt } \, \sum_{\bi \in \T_{\bi_0,t}, \abs{\bi}=p} \, \sum_{m \le q \le p}  \, k(\beta)^{C_1 +m -q}\, \frac{C_2 k(\beta)^{p}}{t^\beta \sqrt{n_t}} \, \frac{k(\beta)^{q-m}}{e^{\chi(p-m)}} \bigg]
\end{align*}

\begin{align*}
~\le&~ \E \bigg[    \sum_{m \le \nt} \, \sum_{\pt \le p \le \nt } \, \sum_{\bi \in \T_{\bi_0,t}, \abs{\bi}=p} \, \frac{C_2 k(\beta)^{p+C_1}}{t^\beta \sqrt{n_t}} \, \frac{(p-m)}{e^{\chi(p-m)}} \bigg] \\
~\le&~     \sum_{m \le \nt} \, \sum_{\pt \le p \le \nt } \, k(\beta)^{C_1 + m -p} \frac{C_2 k(\beta)^{p+C_1}}{t^\beta \sqrt{n_t}} \, \frac{(p-m)}{e^{\chi(p-m)}} \\
~=&~  \frac{C_2}{t^\beta \sqrt{n_t}} \,     \sum_{m \le \nt} \, \sum_{\pt \le p \le \nt }  \, \frac{ k(\beta)^{m+2C_1} (p-m)}{e^{\chi(p-m)}}.
\end{align*}
Now, we have to split the remaining summation, depending on which lower bound for $p$ we are going to use. Before we do so, we note that since $p-m \ge C_1$, we can estimate
$$ \frac{p-m}{e^{\chi (p-m)}} \le \frac{1}{e^{\eta(p-m)}}$$
for some $0 < \eta < \chi$, as soon as $C_1$ is large enough.
\begin{align*}
&~  \frac{C_2}{t^\beta \sqrt{n_t}} \,     \sum_{m \le \nt} \, \sum_{\pt \le p \le \nt }  \, \frac{ k(\beta)^{m+2C_1} (p-m)}{e^{\chi(p-m)}}  \\
\le&~  \frac{C_2}{t^\beta \sqrt{n_t}} \,  \bigg[    \sum_{m +C_1 \le n_t-\sqrt{n_t} } \, \sum_{n_t-\sqrt{n_t} \le p \le \nt }  \, \frac{ k(\beta)^{m+2C_1}}{e^{\eta(p-m)}} ~+~ \sum_{ n_t-\sqrt{n_t} < m+C_1 \le \nt } \, \sum_{m+C_1 \le p \le \nt }  \, \frac{ k(\beta)^{m+2C_1}}{e^{\eta(p-m)}} \bigg] \\
\le&~  \frac{C_2}{t^\beta \sqrt{n_t}} \,  \bigg[    \sum_{m +C_1 \le n_t-\sqrt{n_t} } \,  \frac{ k(\beta)^{m+2C_1}}{e^{\eta[C_1 + n_t - \sqrt{n_t}-(m+C_1)]}} ~+~ \sum_{ n_t-\sqrt{n_t} < m+C_1 \le \nt } \,  \frac{ k(\beta)^{m+2C_1}}{e^{\eta C_1}} \bigg] \\
\le&~  \frac{C_2 k(\beta)^{2 C_1}}{t^\beta \sqrt{n_t} } \,  \bigg[  \frac{1}{e^{\eta C_1}}   \sum_{l = 0}^\infty \,  \frac{1}{e^{\eta l}} ~+~ \frac{\sqrt{n_t}}{2} \frac1{e^{\eta C_1}} \bigg] ~=~  \frac{C_2 k(\beta)^{2 C_1}}{t^\beta e^{\eta C_1}} \,  \bigg[    \frac{1}{ \sqrt{n_t}(1-e^{-\eta})} ~+~  \frac1{2} \bigg]
\end{align*}
For $t$ large enough, the factor in the brackets becomes smaller than one, and we obtain the assertion.
\end{proof}

\section{Proof of the main theorem}\label{subsect:proofmthm}
The proof is just a consequence of the previous results. Lemma \ref{lem:estimate} provides  lower estimates of $\P{\skalar{u,X}>t}$, that is \eqref{eq:estimate} in terms of a subset $\W$ of $\T$ and probabilities $\P{V_{\bi,t}}$, $\P{W_{\bi,\bi',t}}$ for $\bi,\bi'\in\W$. Estimates of those probabilities were given in Lemmas  \ref{lem:estimate Vn} and \ref{lem64}, and the set $\W$ was defined as the beginning of Section \ref{sec: combinatorics}. In view of Lemmas \ref{lem:sumVi} and \ref{lem:sumWii} we obtain
$$  \P{\skalar{u,X}>t} ~\ge~ \kappa \, k(\beta)^{C_1} \, \bigg( \frac{D_1 }{2 C_1} - \frac{C_2}{(\E N)^{C_1} e^{\eta C_1}}\bigg) \, t^{-\beta}.$$
Finally, choosing a large $C_1$ such that the last constant is positive we conclude the result.


\end{document}